\documentclass{amsart}

\title[On the Geometry of Conformally Stationary Spaces]{On the Geometry of Conformally Stationary Lorentz Spaces}

\usepackage{graphicx,pstricks,pst-plot}

\newtheorem{theorem}{Theorem}[section]
\newtheorem{lemma}[theorem]{Lemma}
\newtheorem{proposition}[theorem]{Proposition}
\newtheorem{corollary}[theorem]{Corollary}
\theoremstyle{definition}
\newtheorem{definition}[theorem]{Definition}
\theoremstyle{example}
\newtheorem{example}[theorem]{Example}

\theoremstyle{remark}
\newtheorem{remark}[theorem]{Remark}

\numberwithin{equation}{section}

\author{F. Camargo}
\address{Departamento de Matem\'atica e Estat\'{\i}stica,
Universidade Federal de Campina Grande, Campina Grande,
Para\'{\i}ba, Brazil. 58109-970} \email{fernandaecc@dme.ufcg.edu.br}

\author{A. Caminha}
\address{Departamento de Matem\'atica, Universidade Federal do Cear\'a, Fortaleza,
Cear\'a, Brazil. 60455-760} \email{antonio.caminha@gmail.com}

\author{H. de Lima}
\address{Departamento de Matem\'atica e Estat\'{\i}stica,
Universidade Federal de Campina Grande, Campina Grande,
Para\'{\i}ba, Brazil. 58109-970} \email{henrique@dme.ufcg.edu.br}

\author{M. Vel\'asquez}
\address{Departamento de Matem\'atica, Universidade Federal do Cear\'a, Fortaleza,
Cear\'a, Brazil. 60455-760} \email{marcolazarovelasquez@gmail.com}



\thanks{The second author is partially supported by CNPq, Brazil. The third author is partially supported by CNPq/FAPESQ/PPP, Brazil.
The last author is supported by CAPES, Brazil.}

\begin{document}

\maketitle

\begin{abstract}
In this paper we study several aspects of the geometry of
conformally stationary Lorentz manifolds, and particularly of GRW
spaces, due to the presence of a closed conformal vector field. More
precisely, we begin by extending to these spaces a result of J.
Simons on the minimality of cones in Euclidean space, and apply it
to the construction of complete, noncompact maximal submanifolds of
both de Sitter and anti-de Sitter spaces. Then we state and prove
very general Bernstein-type theorems for spacelike hypersurfaces in
conformally stationary Lorentz manifolds, one of which not assuming
the hypersurface to be of constant mean curvature. Finally, we study
the strong $r$-stability of spacelike hypersurfaces of constant
$r$-th mean curvature in a conformally stationary Lorentz manifold
of constant sectional curvature, extending previous results in the
current literature.
\end{abstract}

\section{Introduction}

An important class of Lorentz manifolds is formed by the so-called stationary Lorentz manifolds. Following~\cite{Wald:84}, Chapter $6$, we say that a Lorentz manifold is stationary if there exists a one-parameter group of isometries whose orbits are timelike curves; for spacetimes, this group of isometries expresses time translation symmetry. From the mathematical viewpoint, a stationary Lorentz manifold is simply a Lorentz manifold furnished with a timelike Killing vector field, and a natural generalization is a {\em conformally stationary} Lorentz manifold, i.t., one furnished with a timelike {\em conformal} vector field. Our interest in conformally stationary Lorentz manifolds is due to the fact that, under an appropriate conformal change of metric, the conformal vector field turns into a Killing one, so that the new Lorentz manifold is now
stationary.

Our aim in this work is to understand the geometry of immersed submanifolds of conformally stationary Lorentz manifolds furnished with a closed conformal vector field, and we do this by approaching three different kinds of problems: the construction of examples of maximal submanifolds, the obtainance of general Bernstein-type results and the derivation of suitable criteria for $r$-stability. 

First of all (cf. Theorem~\ref{thm:maximal submanifolds}), we extend a classical theorem of Simons~\cite{Simons:68} to conformally stationary Lorentz manifolds and apply it to build maximal Lorentz immersions whenever the ambient space either is of constant sectional curvature or has vanishing Ricci curvature in the direction of the conformal vector field; in particular, we provide a geometrical construction for maximal immersions into both the anti-de Sitter and de Sitter spaces (cf. Corollaries~\ref{coro:maximal submanifold in the anti-de Sitter space} and~\ref{coro:maximal submanifold in de Sitter space}).

Related to Bernstein-type results, we study complete spacelike hypersurfaces, not necessarily of constant mean curvature, immersed into a conformally stationary Lorentz manifold of nonnegative Ricci curvature, furnished with homothetic non-parallel vector fields. By asking the second fundamental form of the hypersurface to be bounded and imposing a natural restriction on the projection of the conformal vector field of the ambient space, we classify such hypersurfaces in Theorems~\ref{thm:Bernstein-type} and \ref{thm:Bernstein-type 2}. In particular, we classify complete, finitely punctured spacelike radial graphs over $\mathbb R^n$ or $\mathbb H^n$ in $\mathbb L^{n+1}$, thus extending a classical result of J. Jellett~\cite{Jellett:53} to the Lorentz context.

Finally, in the last section we derive in Theorem~\ref{thm:r_estability} a sufficient criterion for strong $r$-stability of spacelike hypersurfaces of constant $r-$th mean curvature. This result extends, to the class of conformally stationary Lorentz manifolds, previous ones obtained in~\cite{BBC:08} and~\cite{Camargo:09} in the context of Generalized Robertson-Walker spacetimes.

\section{Conformally stationary Lorentz manifolds}

As in the previous section, let $\overline M^{n+1}$ be a Lorentz manifold. We recall that a vector field $V$ on $\overline M^{n+1}$ is conformal if
\begin{equation}
\mathcal L_V\langle\,\,,\,\,\rangle=2\psi\langle\,\,,\,\,\rangle
\end{equation}
for some function $\psi\in C^{\infty}(\overline M)$, where $\mathcal L$ stands for the Lie derivative of the Lorentz metric of $\overline M$; the function $\psi$ is the conformal factor of $V$. Any Lorentz manifold $\overline M^{n+1}$ possessing a globally defined, timelike conformal vector field is said to be conformally stationary.

Since $\mathcal L_V(X)=[V,X]$ for all $X\in\mathfrak X(\overline M)$, the tensorial character of $\mathcal L_V$ shows that $V\in\mathfrak X(\overline M)$ is conformal if and only if
\begin{equation}\label{eq:1.1}
\langle\overline\nabla_XV,Y \rangle+\langle X,\overline\nabla_YV\rangle=2\psi\langle X,Y\rangle,
\end{equation}
for all $X,Y\in\mathfrak X(\overline M)$. In particular, $V$ is Killing if and only if $\psi\equiv 0$, and (\ref{eq:1.1}) gives
$$\psi=\frac{1}{n+1}\text{div}_{\overline M}V.$$

Suppose that our conformally stationary spacetime $\overline
M^{n+1}$ is endowed with a closed conformal timelike vector field
$V$ with conformal factor $\psi$, i.e., one for which
\begin{equation}\label{eq:closed conformal}
\overline\nabla_XV=\psi X
\end{equation}
for all $X\in\mathfrak X(\overline M)$. We recall that such a $V$ is parallel if $\psi$ vanishes identically and homothetic if $\psi$ is constant\footnote{Here we diverge a little bit from other papers (e.g.~\cite{Kuhnel:97}), where homothetic means just conformal with constant conformal factor. The reason is economy: to avoid constantly writing {\em closed and homothetic}.}.

If $V$ has no singularities on an open set $\mathcal U\subset\overline M$, then the distribution $V^{\bot}$ on
$\mathcal U$ of vector fields orthogonal to $V$ is integrable, for if $X,Y\in V^{\bot}$, then
$$\langle[X,Y],V\rangle=\langle\overline\nabla_XY-\overline\nabla_YX,V\rangle=-\langle Y,\overline\nabla_XV\rangle+\langle X,\overline\nabla_YV\rangle=0.$$
We let $\Xi$ be a leaf of $V^{\bot}$ furnished with the induced metric, and $D$ its Levi-Civita connection.

From (\ref{eq:closed conformal}) we get
\begin{equation}\label{eq:gradient conformal}
\overline\nabla \langle V,V \rangle=2\psi V,
\end{equation}
so that $\langle V,V\rangle$ is constant on connected leaves of $V^{\bot}$. Computing covariant derivatives in (\ref{eq:gradient conformal}), we have
$$(\text{Hess}_{\overline M}\langle V,V \rangle)(X,Y)=2X(\psi)\langle V,Y \rangle + 2\psi^2\langle X,Y \rangle.$$
However, since both $\text{Hess}_{\overline M}$ and the metric are symmetric tensors, we get
$$X(\psi)\langle V,Y\rangle =Y(\psi)\langle V,X\rangle$$
for all $X,Y\in\mathfrak X(\overline M)$. Taking $Y=V$ we then arrive at
\begin{equation}\label{eq:gradient conformal factor}
\overline\nabla \psi= \frac{V(\psi)}{\langle V,V \rangle}V=-\nu(\psi)\nu ,
\end{equation}
where $\nu=\frac{V}{\sqrt{-\langle V,V \rangle}}$. Hence, $\psi$ is also constant on connected leaves of $V^{\bot}$. If $\Xi$ is such a leaf and $S_{\Xi}$ denotes its shape operator with respect to $\nu$, we get
\begin{equation}\label{eq:shape operator of a leaf in the perp distribution}
S_{\Xi}(X)=-\overline\nabla_X\nu=\psi X,
\end{equation}
and hence $\Xi$ is an umbilical hypersurface of $\overline M$.

Now we need the following

\begin{lemma}
If $\eta$ is another closed conformal vector field on $\overline M$ and $U=\eta+\langle\eta,\nu\rangle\nu$, then $U$ is closed conformal on $\Xi$, with conformal factor $\psi_U=\psi_{\eta}+\psi_{\nu}\langle\eta,\nu\rangle$.
\end{lemma}

\begin{proof}
For $Z\in\mathfrak X(\Xi)$, it follows from $\langle Z,\nu\rangle=0$
that
\begin{equation}\label{eq:conformality of the projected field}
 \begin{split}
D_ZU&\,=(\overline\nabla_ZU)^{\top}=\overline\nabla_ZU+\langle\overline\nabla_ZU,\nu\rangle\nu\\
&\,=\overline\nabla_Z(\eta+\langle\eta,\nu\rangle\nu)+\langle\overline\nabla_Z(\eta+\langle \eta,\nu\rangle\nu),\nu\rangle\nu\\
&\,=\overline\nabla_Z\eta+Z\langle\eta,\nu\rangle\nu+\langle \eta,\nu\rangle\overline\nabla_Z\nu+\langle\overline\nabla_Z\eta,\nu\rangle\nu\\
&\,\,\,\,\,\,\,+Z\langle\eta,\nu\rangle\langle\nu,\nu\rangle\nu+\langle \eta,\nu\rangle\langle\overline\nabla_Z\nu,\nu\rangle\nu\\
&\,=\overline\nabla_Z\eta+\langle\eta,\nu\rangle\overline\nabla_Z\nu+\langle\overline\nabla_Z\eta,\nu\rangle\nu\\
&\,=\psi_{\eta}Z+\langle\eta,\nu\rangle\psi_{\nu}Z+\langle\psi_{\eta}Z,\nu\rangle\nu\\
&\,=(\psi_{\eta}+\langle\eta,\nu\rangle\psi_{\nu})Z.
 \end{split}
\end{equation}
\end{proof}

\begin{example}
Let $\mathbb L^{n+1}$ be the $(n+1)-$dimensional Lorentz space with
its usual scalar product $\langle\cdot,\cdot\rangle$, with respect
to the quadratic form $q(x)=\sum_{i=1}^nx_i^2-x_{n+1}^2$. For $n>2$,
the $n-$dimensional de Sitter space is the hyperquadric
$$\mathbb S_1^n=\{x\in\mathbb L^{n+1};\,\langle x,x\rangle=1\},$$
and also the Lorentz simply-connected space form of constant sectional curvature identically equal to $1$. The previous proposition teaches how to geometrically build closed conformal vector field on $\mathbb S_1^n$: since $\nu\in\mathfrak X(\mathbb L^{n+1})$ given by $\nu(x)=x$ is homothetic, choose a parallel $\eta\in\mathfrak X(\mathbb L^{n+1})$ and project it orthogonally onto the $\mathbb S_1^n$.
\end{example}

\begin{example}
Let $R^{n+1}_2$ denote $R^{n+1}$ furnished with the scalar product correspondent to the quadratic form $q(x)=\sum_{i=1}^{n-1}x_i^2-x_n^2-x_{n+1}^2$. For $n>1$, the $n-$dimensional anti-de Sitter space is the hyperquadric
$$\mathbb H_1^n=\{x\in\mathbb R^{n+1}_2;\,\langle x,x\rangle=-1\},$$
and also the Lorentz simply-connected space form of constant sectional curvature identically equal to $-1$. A construction similar to that of the previous example can obviously be made for $\mathbb H_1^n$.
\end{example}

Following~\cite{ABC:03}, a particular class of conformally stationary spacetimes is that of Generalized Robertson-Walker spaces (GRW for short), namely, warped products $\overline M^{n+1}=-I\times_{\phi}F^n$, where $I\subset\mathbb R$ is an open interval with the metric $-dt^2$, $F^n$ is an $n$-dimensional Riemannian manifold and $\phi:I\rightarrow\mathbb R$ is positive and smooth. For such a space, if $\pi_I:\overline M^{n+1}\rightarrow I$ is the canonical projection onto $I$, then the vector field
$$V=(\phi\circ\pi_I)\partial_t$$
is a conformal, timelike and closed, with conformal factor $\psi=\phi'\circ\pi_I$, where the prime denotes differentiation with respect to $t$. Moreover (cf.~\cite{Montiel:99}), for $t_0\in I$, the (spacelike) leaf $\Xi_{t_0}^n=\{t_0\}\times F^n$ is totally umbilical, with umbilicity factor $-\frac{\phi'(t_0)}{\phi(t_0)}$ with respect to the future-pointing unit normal vector field.

\begin{remark}
Conversely, let $\overline M$ be a general conformally stationary Lorentz manifold with closed conformal vector field $V$. If $p\in\overline M$ and $\Xi_p$ is the leaf of $V^{\bot}$ passing through $p$, then we can find a neighborhood $\mathcal U_p$ of $p$ in $\Xi_p$ and an open interval $I\subset\mathbb R$ containing $0$ such that the flow $\Psi$ of $V$ is defined on $\mathcal U_p$ for every $t\in I$. Besides, if $\overline M$ is timelike geodesically complete, S. Montiel~\cite{Montiel:99} proved that
$$\begin{array}{rccc}
\varphi:&\mathbb R\times\Xi_p&\longrightarrow&\overline M\\
&(t,p)&\mapsto&\Psi(t,p)
\end{array}$$
is a global parametrization on $\overline M$, such that $\overline M$ is isometric to the GRW $-\mathbb R\times_{\phi}\Xi_p$, where $\phi(q)=t\Leftrightarrow\Psi(-t,q)\in\Xi_p$.
\end{remark}

Let $x:M^n\rightarrow\overline M^{n+1}$ a spacelike immersion, that
is, the induced metric via $x$ on $M^n$ is Riemannian. According
to~\cite{ABC:03}, if $\overline M^{n+1}=-I\times_{\phi}F^n$ is a GRW
and $x:M^n\rightarrow\overline M^{n+1}$ is a complete spacelike
hypersurface such that $\phi\circ\pi_I$ is bounded on $M^n$, then
$\pi_F\big|_M:M^n\rightarrow F^n$ is necessarily a covering map. In
particular, if $M^n$ is closed, then $F^n$ is automatically closed.

Also, recall (cf. Chapter $7$ of~\cite{ONeill:83}) that a GRW as above has constant sectional curvature $c$ if and only if $F$ has constant sectional curvature $k$ and the warping function $\phi$ satisfies the ODE
\begin{equation}\label{eq:ode for constant sectional curvature}
 \frac{\phi''}{\phi}=c=\frac{(\phi')^2+k}{\phi^2}.
\end{equation}

\begin{example}
The de Sitter space $\mathbb S_1^{n+1}$ is an important particular example of GRW of constant sectional curvature $1$. In fact, it follows from {\rm(\ref{eq:ode for constant sectional curvature})} and the classification of simply connected Lorentz space forms that
$$\mathbb S_1^{n+1}\simeq-\mathbb R\times_{\cosh t}\mathbb S^n,$$
where $\mathbb S^n$ is the standard $n-$dimensional unit sphere in
Euclidean space. Hence, the vector field $V=(\sinh t)\partial_t$ is
a timelike closed conformal one. The {\em equator} of $\mathbb
S_1^{n+1}$ is the slice $\{0\}\times\mathbb S^n$, and the points
$(t,p)\in\mathbb S_1^{n+1}$ with $t<0$ {\rm(}resp. $t>0${\rm)} are
said to form the chronological past {\rm(}resp. future{\rm)} of
$\mathbb S_1^{n+1}$.
\end{example}

\begin{example}
Another important example is given by the anti-de Sitter space $\mathbb H_1^{n+1}$. Invoking once more {\rm(\ref{eq:ode for constant sectional curvature})} and the classification of simply connected Lorentz space forms, we get the isometry
$$\mathbb H_1^{n+1}\simeq-(0,\pi)\times_{\sin t}\mathbb H^n,$$
where $\mathbb H^n$ is the $n-$dimensional hyperbolic space. Therefore, the vector field $V=(\sin t)\partial_t$ is timelike closed and conformal in $\mathbb H_1^{n+1}$.
\end{example}

\section{Maximal submanifolds of conformally stationary spaces}

In this section we generalize to conformally stationary Lorentz
manifolds a theorem of J. Simons~\cite{Simons:68}, which shows how
one can build isometric immersions with parallel mean curvature in
$\mathbb R^{n+k+1}$ from minimal immersions
$\varphi:M^n\rightarrow\mathbb S^{n+k}$.

As in the previous section, let $\overline M^{n+k+1}$ be an
$(n+k+1)-$dimensional conformally stationary Lorentz manifold, with
closed conformal vector field $V$ of conformal factor $\psi$. If
$V\neq 0$ on $\overline M$, we saw that the orthogonal distribution
$V^{\bot}$ is integrable, with totally umbilical leaves. Therefore,
if $\Xi^{n+k}$ be such a leaf, then it is an umbilical spacelike
hypersurface of $\overline M^{n+k+1}$ and
$\nu=\frac{V}{\sqrt{-\langle V,V\rangle}}$ is a global unit normal
timelike vector field on it.

Let $\varphi:M^n\rightarrow\Xi^{n+k}$ be an isometric immersion, where $M^n$ is a compact Riemannian manifold. If $\Psi$ denotes the flow of $V$, the compactness of $M^n$ guarantees the existence of $\epsilon>0$ such that $\Psi$ is defined on $(-\epsilon,\epsilon)\times\varphi(M)$, and the map
\begin{equation}\label{eq:immersion Phi}
 \begin{array}{rccc}
 \Phi:&(-\epsilon,\epsilon)\times M^n&\longrightarrow&\overline{M}^{n+k+1}\\
&(t,q)&\mapsto&\Psi(t,\varphi(q))
  \end{array}
\end{equation}
is also an immersion. Furnishing $(-\epsilon,\epsilon)\times M^n$ with the metric induced by $\Phi$, we turn it into a Lorentz manifold an $\Phi$ into an isometric immersion such that $\Phi_{|\{0\}\times M^n}=\varphi$.

Finally, letting $\text{Ric}_{\overline{M}}$ denote the field of
self-adjoint operators associated to the Ricci tensor of
$\overline{M}$, we get the above-mentioned generalization of Simons'
result (see also~\cite{Caminha:09} for the Riemannian case).

\begin{theorem}\label{thm:maximal submanifolds}
In the above notations, let $\psi\neq 0$ on $\varphi(M)$. If $\overline M$ has constant sectional curvature or $\text{\rm Ric}_{\overline M}(V)=0$, then the following are equivalent:
\begin{enumerate}
 \item[$(a)$] $\varphi$ is maximal.
 \item[$(b)$] $\Phi$ is maximal.
 \item[$(c)$] $\Phi$ has parallel mean curvature.
\end{enumerate}
\end{theorem}

\begin{proof}
Fix $p\in M$ and, on a neighborhood $\Omega$ of $p$ in $M$, an orthonormal frame $\{e_1,\ldots,e_n,\eta_1,\ldots,\eta_k\}$ adapted to $\varphi$, such that $\{e_1,\ldots,e_n\}$ is geodesic at $p$.

If $E_1,\ldots,E_n,N_1,\ldots,N_k$ are the vector fields on $\Psi((-\epsilon,\epsilon)\times\Omega)$ obtained from the $e_i$'s and $\eta_{\beta}$'s by parallel transport along the integral curves of $V$ that intersect $\Omega$, it follows that $\{E_1,\ldots,E_n,\nu,N_1,\ldots,N_k\}$ is an orthonormal frame on $\Psi((-\epsilon,\epsilon)\times\Omega)$, adapted to the immersion (\ref{eq:immersion Phi}).

Let $\overline\nabla$ be the Levi-Civita connection of $\overline M$ and $\overline H$ the mean curvature vector of $\Phi$. It follows from the closed conformal character of $V$ that, on $\Phi((-\epsilon,\epsilon)\times\Omega)$,
\begin{equation}\label{eq:auxiliar 0 para cone com vetor curvatura media paralelo}
\overline H=\frac{1}{n+1}(\overline\nabla_{E_i}E_i+\overline\nabla_{\nu}\nu)^{\bot}
=\frac{1}{n+1}(\tilde\nabla_{E_i}E_i)^{\bot},
\end{equation}
where $\bot$ denotes orthogonal projection on $T\Phi((-\epsilon,\epsilon)\times\Omega)^{\bot}$.

In order to compute $\overline\nabla_{E_i}E_i$ along the integral curve that passes through $p$, note that
\begin{equation}\label{eq:auxiliar 2 para cone com vetor curvatura media paralelo}
 \langle\overline\nabla_{E_i}E_i,V\rangle=-\langle E_i,\overline\nabla_{E_i}V\rangle=-n\psi.
\end{equation}

Now, if $\overline R$ stands for the curvature operator of $\overline M$, observe that
\begin{equation}\label{eq:auxiliar 3 para cone com vetor curvatura media paralelo}
 \begin{split}
 \frac{d}{dt}\langle\overline\nabla_{E_i}E_i,E_k\rangle&\,=\langle\overline\nabla_V\overline\nabla_{E_i}E_i,E_k\rangle\\
&\,=\langle\overline R(V,E_i)E_i,E_k\rangle+\langle\overline\nabla_{E_i}\overline\nabla_VE_i,E_k\rangle+\langle\overline\nabla_{[V,E_i]}E_i,E_k\rangle\\
&\,=\langle\text{Ric}_{\overline{M}}(V),E_k\rangle-\langle\overline\nabla_{\overline\nabla_{E_i}V}E_i,E_k\rangle\\
&\,=-\psi\langle\overline\nabla_{E_i}E_i,E_k\rangle.
 \end{split}
\end{equation}
Note that, in the last equality, we used the fact that either $\overline M$ has constant sectional curvature or $\text{Ric}_{\overline M}(V)=0$ to conclude that $\langle\text{Ric}_{\overline M}(V),E_k\rangle=0$.

Let $D$ and $\nabla$ respectively denote the Levi-Civita connections of $\Xi^{n+k}$ and $M^n$. Since $\{e_1,\ldots,e_n\}$ is geodesic at $p$ (on $M$), it follows that
\begin{equation}\label{eq:auxiliar 4 para cone com vetor curvatura media paralelo}
 \langle\overline\nabla_{E_i}E_i,E_k\rangle_p=\langle D_{e_i}e_i,e_k\rangle_p=\langle(D_{e_i}e_i)^{\bot}+\nabla_{e_i}e_i,e_k\rangle_p=0.
\end{equation}
Therefore, solving the Cauchy problem formed by (\ref{eq:auxiliar 3 para cone com vetor curvatura media paralelo}) and (\ref{eq:auxiliar 4 para cone com vetor curvatura media paralelo}), we get
\begin{equation}\label{eq:auxiliar 5 para cone com vetor curvatura media paralelo}
\langle\overline\nabla_{E_i}E_i,E_k\rangle_{\Psi(t,p)}=0,\,\,\forall\,\,|t|<\epsilon.
\end{equation}

Analogously to (\ref{eq:auxiliar 3 para cone com vetor curvatura media paralelo}), we get
\begin{equation}\label{eq:auxiliar 6 para cone com vetor curvatura media paralelo}
  \frac{d}{dt}\langle\overline\nabla_{E_i}E_i,N_{\beta}\rangle=-\psi\langle\overline\nabla_{E_i}E_i,N_{\beta}\rangle.
\end{equation}
On the other hand, letting $A_{\beta}:T_pM\rightarrow T_pM$ denote the shape operator of $\varphi$ in the direction of $\eta_{\beta}$ and writing $A_{\beta}e_i=h^{\beta}_{ij}e_j$, we have
\begin{equation}\label{eq:auxiliar 7 para cone com vetor curvatura media paralelo}
 \langle\overline\nabla_{E_i}E_i,N_{\beta}\rangle_p=\langle D_{e_i}e_i,\eta_{\beta}\rangle_p=\langle A_{\beta}e_i,e_i\rangle_p=h^{\beta}_{ii}.
\end{equation}
Solving the Cauchy problem formed by (\ref{eq:auxiliar 6 para cone com vetor curvatura media paralelo}) and (\ref{eq:auxiliar 7 para cone com vetor curvatura media paralelo}), we get
\begin{equation}\label{eq:auxiliar 8 para cone com vetor curvatura media paralelo}
 \langle\overline\nabla_{E_i}E_i,N_{\beta}\rangle_{\Psi(t,p)}=h^{\beta}_{ii}\exp\left(-\int_0^t\psi(s)ds\right).
\end{equation}

It finally follows from (\ref{eq:auxiliar 2 para cone com vetor curvatura media paralelo}), (\ref{eq:auxiliar 5 para cone com vetor curvatura media paralelo}) and (\ref{eq:auxiliar 8 para cone com vetor curvatura media paralelo}) that, at the point $(t,p)$,
\begin{eqnarray*}
 \overline\nabla_{E_i}E_i&=&\langle\overline\nabla_{E_i}E_i,E_k\rangle E_k-\langle\overline\nabla_{E_i}E_i,\nu\rangle\nu+\langle\overline\nabla_{E_i}E_i,N_{\beta}\rangle N_{\beta}\\
&=&\frac{n\psi}{\langle V,V\rangle}V+\exp\left(-\int_0^t\psi(s)ds\right)h^{\beta}_{ii}N_{\beta}.
\end{eqnarray*}
Therefore, (\ref{eq:auxiliar 0 para cone com vetor curvatura media paralelo}) gives us
\begin{equation}\label{eq:auxiliar 9 para cone com vetor curvatura media paralelo}
 \overline H=\frac{1}{n+1}\exp\left(-\int_0^t\psi(s)ds\right)h^{\beta}_{ii}N_{\beta}.
\end{equation}

Let us finally establish the equivalence of (a), (b) and (c), observing that (b) $\Rightarrow$ (c) is always true.\\

\noindent (a) $\Rightarrow$ (b): if $\varphi$ is maximal, we have $h^{\beta}_{ii}=0$ for all $1\leq\beta\leq k$, and it follows from (\ref{eq:auxiliar 9 para cone com vetor curvatura media paralelo}) that $\overline{H}=0$.\\

\noindent (c) $\Rightarrow$ (a): if $\overline\nabla^{\bot}\overline H=0$, then, along the integral curve of $V$ that passes through $p$, the parallelism of the $N_{\beta}$ gives
$$0=\overline\nabla_V^{\bot}\overline H=\left(\frac{D\overline H}{dt}\right)^{\bot}=-\frac{\psi(t)}{n+1}\exp\left(-\int_0^t\psi(s)ds\right)h^{\beta}_{ii}N_{\beta}.$$
However, since $\psi\neq 0$ on $\varphi(M)$, it follows from the
above equality that $h^{\beta}_{ii}=0$ at $p$ for all
$1\leq\beta\leq k$, so that $\varphi$ is maximal at $p$.
\end{proof}

The following corollaries are immediate.

\begin{corollary}\label{coro:maximal submanifold in GRW I}
Let $I\subset(0,+\infty)$ and $t_0\in I$. In the GRW space $-I\times_t\mathbb F^{n+k}$, if $\varphi:M^n\rightarrow\{t_0\}\times F^{n+k}$ is an isometric immersion and $\Phi$ is the canonical immersion of $-I\times_tM^n$ into $-I\times_tF^{n+k}$, then $\varphi$ is maximal if and only if $\Phi$ is maximal.
\end{corollary}

\begin{proof}
 Since $V=t\partial_t$, Corollary $7.43$ of~\cite{ONeill:83}, $\text{Ric}_{\overline M}(V)=0$ in this case.
\end{proof}

\begin{corollary}\label{coro:maximal submanifold in the anti-de Sitter space}
Let $\varphi:M^n\rightarrow\mathbb S^{n+k}$ be an $n-$dimensional
submanifold of some round sphere $\mathbb S^{n+k}$ of the
$(n+k+1)-$dimensional de Sitter space $\mathbb S_1^{n+k+1}$. If
$\varphi(M)$ is contained in the chronological past {\rm(}resp.
future{\rm)} of $\mathbb S_1^{n+k+1}$, then $M^n$ is maximal in
$\mathbb S^{n+k}$ if and only if the union of the segments of the
integral curves of $\partial_t$ contained in the chronological past
{\rm(}resp. future{\rm)} of $\mathbb S_1^{n+k+1}$ and passing
through points of $\varphi(M)$ is maximal in $\mathbb S_1^{n+k+1}$.
\end{corollary}

\begin{corollary}\label{coro:maximal submanifold in de Sitter space}
Let $\varphi:M^n\rightarrow\mathbb H^{n+k}$ be an $n-$dimensional
submanifold of some hyperbolic space $\mathbb H^{n+k}$ of the
$(n+k+1)-$dimensional anti-de Sitter space $\mathbb H_1^{n+k+1}$.
Then $M^n$ is maximal in $\mathbb H^{n+k}$ if and only if the union
of the segments of the integral curves of $\partial_t$ that pass
through points of $\varphi(M)$ is maximal in $\mathbb H_1^{n+k+1}$.
\end{corollary}

\section{Bernstein-type Theorems}

We continue to employ the notations of the previous sections, i.e., $\overline M$ is conformally stationary with closed conformal vector field $V$. However, we let $\psi_V$ be the conformal factor of $V$.

From now on, we let $x:M^n\rightarrow\overline M^{n+1}$ be a connected, complete, oriented spacelike hypersurface, $N$ be a unit normal vector field which orients $M$ and has the same time-orientation of $V$, and $A$ and $H$ be respectively the shape operator and the mean curvature of $M$ with respect to $N$.

If $f_V:M\rightarrow\mathbb R$ is given by $f_V=\langle V,N\rangle$ then $f_V$ is negative on $M$. On the other hand, standard computations (cf.~\cite{BBC:08}) give
\begin{equation}\label{eq:Gradient of f nu}
 \nabla f_V=-A(V^{\top})
\end{equation}
and
\begin{equation}\label{eq:Laplacian of f nu}
\Delta f_V=nV^{\top}(H)+\left\{\text{\rm Ric}_{\overline M}(N,N)+|A|^2\right\}f_V+n\left\{H\psi+N(\psi)\right\},
\end{equation}
where $(\,\,)^{\top}$ stands for orthogonal projection onto $M$.

If $W$ is another closed conformal vector field on $\overline M$, with conformal factor $\psi_W$, and $g:M\rightarrow\mathbb R$ is given by $g=\langle V,W\rangle$, then another standard computation gives
\begin{equation}\label{eq:Gradient of g}
\nabla g=\psi_VW^{\top}+\psi_WV^{\top}.
\end{equation}
and
\begin{equation}\label{eq:Laplacian of nu inner eta}
 \Delta g=W^{\top}(\psi_V)+V^{\top}(\psi_W)+nH(\psi_Vf_W+\psi_Wf_V)+2n\psi_V\psi_W.
\end{equation}

We are now in position to state and prove the following Bernstein-type general theorem for spacelike hypersurfaces. Observe that we do not require the hypersurface in question to be of constant mean curvature. In what follows, we let $\mathcal L^1(M)$ be the space of Lebesgue integrable functions on $M$.

\begin{theorem}\label{thm:Bernstein-type}
Let $\overline M^{n+1}$ have nonnegative Ricci curvature, $V$ and
$W$ be respectively a parallel and a homothetic nonparallel vector
field on $\overline M^{n+1}$, and
$x:M^n\rightarrow\overline{M}^{n+1}$ be as above. If $|A|$ is
bounded, $|V^{\top}|$ is integrable and $H$ doesn't change sign on
$M$, then:
\begin{enumerate}
\item[$(a)$] $M$ is totally geodesic and the Ricci curvature of $\overline M$ in the direction of $N$ vanishes identically.
\item[$(b)$] If $M$ is noncompact and ${\rm Ric}_M$ is also nonnegative, then $x(M)$ is contained in a leaf of $V^{\bot}$.
\end{enumerate}
\end{theorem}

\begin{proof}\ \

\noindent (a) Since $V$ is parallel and $W$ is homothetic and nonparallel, it follows from (\ref{eq:Gradient of f nu}), (\ref{eq:Gradient of g}), (\ref{eq:Laplacian of f nu}) and (\ref{eq:Laplacian of nu inner eta}) that $\nabla f_V=-A(V^{\top})$, $\nabla g=\psi_WV^{\top}$,
\begin{equation}\label{eq:Laplacian of f nu for parallel nu}
 \Delta f_V=nV^{\top}(H)+(\text{\rm Ric}_{\overline M}(N,N)+|A|^2)f_V
\end{equation}
and
$$\Delta g=nH\psi_Wf_V,$$
with $\psi_W$ being a nonzero constant. Therefore, the assumption $|V^{\top}|\in\mathcal L^1(M)$ guives $|\nabla
g|\in\mathcal L^1(M)$, and the assumption on $H$, together with the fact that $|f_V|>0$ on $M$, assures that $\Delta g$ is either nonnegative or nonpositive on $M$. Therefore, the Corollary on page $660$ of~\cite{Yau:76} gives $\Delta g=0$ on $M$, and hence $H=0$ on $M$.

We now look at (\ref{eq:Laplacian of f nu for parallel nu}), which resumes to
$$\Delta f_V=(\text{\rm Ric}_{\overline M}(N,N)+|A|^2)f_V,$$
and hence doesn't change sign on $M$ too. We also note that the boundedness of $|A|$ on $M$ gives
$$|\nabla f_V|\leq|A||V^{\top}|\in\mathcal L^1(M).$$
Appealing again to the Corollary on page $660$ of~\cite{Yau:76}, we get $\Delta f_V=0$ on $M$, so that
$$\text{\rm Ric}_{\overline M}(N,N)+|A|^2=0$$
on $M$. Since $\text{\rm Ric}_{\overline{M}}(N,N)\geq 0$, then we get $\text{\rm Ric}_{\overline{M}}(N,N)=0$ and $A=0$ on $M$, i.e., $M$ is totally geodesic.\\

\noindent (b) $A=0$ on $M$ gives $\nabla f_V=0$ on $M$, so that $f_V=\langle V,N\rangle$ is constant and nonzero on $M$. However, $|V|^2$ is constant on $\overline M$ (since $V$ is parallel) and
$$|V^{\top}|^2=|V|^2+\langle V,N\rangle^2,$$
so that $|V^{\top}|$ is also constant on $M$. Therefore,
$$+\infty>\int_M|V^{\top}|dM=|V^{\top}|\,{\rm Vol}(M).$$
But since $M$ is noncompact and has nonnegative Ricci curvature, another theorem of Yau (Theorem $7$ of~\cite{Yau:76}) gives ${\rm Vol}(M)=+\infty$, and hence the only possibility is $|V^{\top}|=0$. Therefore, Cauchy-Schwarz inequality gives that $V$ is parallel to $N$, and $x(M)$ is contained in a leaf of $V^{\bot}$.
\end{proof}

For the next result we need a small refinement of the analytical tool of Yau's result used in the above proof. We quote it below, refering the reader to~\cite{Caminha:09} for a proof.

\begin{lemma}\label{prop:first corollary of Yau 76}
Let $X\in\mathfrak X(M)$ be such that ${\rm div}_MX$ doesn't change sign on $M$. If $|X|\in\mathcal L^1(M)$, then ${\rm div}X=0$ on $M$.
\end{lemma}

\begin{theorem}\label{thm:Bernstein-type 2}
Let $\overline M$ have nonnegative Ricci curvature, $V$ be a
homothetic vector field on $\overline{M}^{n+1}$, and
$x:M^n\rightarrow\overline{M}^{n+1}$ be as before. If $|A|$ is
bounded, $|V^{\top}|$ is integrable and $H$ is constant on $M$, then
$M$ is totally umbilical and the Ricci curvature of $\overline M$ in
the direction of $N$ vanishes identically.
\end{theorem}

\begin{proof}
Since $H$ is constant on $M$ and $\psi_V$ is constant on $\overline M$, (\ref{eq:Laplacian of f nu}) reduces to
$$\Delta f_V=\left\{\text{\rm Ric}_{\overline M}(N,N)+|A|^2\right\}f_V+nH\psi_V.$$
Letting $\{e_1,\ldots,e_n\}$ be a moving frame on $M$, we have
\begin{equation}\label{eq:pulo do gato 1}
 \begin{split}
{\rm div}_M(V^{\top})&\,=\langle\overline\nabla_{e_k}(V+f_VN),e_k\rangle\\
&\,=n\psi_V-f_V\langle A(e_k),e_k\rangle\\
&\,=n\psi_V+nHf_V,
 \end{split}
\end{equation}
so that
\begin{equation}\label{eq:divergence of correct field}
 \begin{split}
{\rm div}_M(\nabla f_V-HV^{\top})&\,=\Delta f_V-nH\psi_V-nH^2f_V\\
&\,=\left\{\text{\rm Ric}_{\overline M}(N,N)+|A|^2-nH^2\right\}f_V.
 \end{split}
\end{equation}

Since $|f_V|>0$ on $M$, ${\rm Ric}_{\overline M}(N,N)\geq 0$ and $|A|^2\geq nH^2$ by Cauchy-Schwarz inequality (with equality if and only if $M$ is totally umbilical), this last expression does not change sign on $M$. Now observe that
$$|\nabla f_V-HV^{\top}|=|-AV^{\top}-HV^{\top}|\leq(|A|+H)|V^{\top}|\in\mathcal L^1(M),$$
so that the previous lemma gives ${\rm div}_M(\nabla f_V-HV^{\top})=0$ on $M$. Back to (\ref{eq:divergence of correct field}), we then get $\text{\rm Ric}_{\overline M}(N,N)=0$ and $|A|^2-nH^2=0$, and hence $M$ is totally umbilical.
\end{proof}

The previous result yields the following corollary on GRW spaces.

\begin{corollary}\label{coro:Bernstein-type for warped products}
Let $I\subset(0,+\infty)$ be an open interval, $F^n$ be an $n-$dimensional, complete oriented Riemannian manifold of nonnegative Ricci curvature, $\overline M^{n+1}=-I\times_tF^n$ and $x:M^n\rightarrow\overline{M}^{n+1}$ as before. If $|A|$ is bounded, $|(t\partial_t)^{\top}|\in\mathcal L^1(M)$ and $H$ is constant on $M$, then $M$ is totally umbilical and the Ricci curvature of $\overline{M}$ in the direction of $N$ vanishes identically. In
particular, if $F$ is closed and has positive Ricci curvature everywhere, then $x(M)\subset\{t_0\}\times F$, for some $t_0\in I$.
\end{corollary}

\begin{proof}
 The first part follows from the theorem. To the second one, if $F$ is closed and has positive Ricci curvature everywhere, then, according to the previous result, the only possible direction for $N$ is that of $t\partial_t$. But if $N$ is parallel to $\partial_t$, then the connectedness of $M$ guarantees that $x(M)$ cannot jump from one leaf $\{t_0\}\times F$ to another.
\end{proof}

In what follows, we let $\mathbb R^n=\{x\in\mathbb L^{n+1};\,x_{n+1}=0\}$ and $\mathbb H^n=\{x\in\mathbb L^{n+1};\,\langle x,x\rangle=-1,\,\,x_{n+1}>0\}$. As a special case of the previous corollary, we get

\begin{corollary}
Let $x:M^n\rightarrow\mathbb L^{n+1}$ be an embedding, such that $x(M)$ is a complete spacelike radial graph over either $\mathbb R^n$ or $\mathbb H^n$, minus $k$ points. If $|A|$ is bounded, $H$ is constant and $p\mapsto|x(p)^{\top}|$ is integrable on $M$, then $k=0$ and $x(M)$ is either a spacelike hyperplane or a translation of $\mathbb H^n$.
\end{corollary}

\begin{remark}\label{rem:sharpness of hypotheses}
The class of examples of Corollary $3.3$ of~\cite{Caminha:06} shows that the hypothesis on the integrability of $p\mapsto|x(p)^{\top}|$ is really necessary.
\end{remark}

\section{$r$-stability of spacelike hypersurfaces}\label{sec:Preliminaries}

For the time being, let $\overline M^{n+1}$ denote a time-oriented Lorentz manifold (i.e., not necessarily conformally stationary) with Lorentz metric $\overline g=\langle\,\,,\,\,\rangle$, volume element $d\overline M$ and semi-Riemannian connection $\overline\nabla$. We consider spacelike hypersurfaces $x:M^n\rightarrow\overline M^{n+1}$, namely, isometric
immersions from a connected, $n-$dimensional orientable Riemannian manifold $M^n$ into $\overline M$. We let $\nabla$ denote the Levi-Civita connection of $M^n$.

If $\overline M$ is time-orientable and $x:M^n\rightarrow\overline M^{n+1}$ is a spacelike  hypersurface, then $M^n$ is orientable (cf.~\cite{ONeill:83}) and one can choose a globally defined unit normal vector field $N$ on $M^n$ having the same time-orientation of $\overline M$; such an $N$ is said to be a {\em future-pointing Gauss map} of $M^n$. If we let $A$ denote the shape operator of $x$ with respect to $N$, then $A$ restricts to a self-adjoint linear map $A_p:T_pM\rightarrow T_pM$ at each $p\in M^n$.

For $1\leq r\leq n$, let $S_r(p)$ denote the $r$-th elementary symmetric function on the eigenvalues of $A_p$, so that one gets $n$ smooth functions $S_r:M^n\rightarrow\mathbb R$ for which $$\det(t\text{Id}-A)=(-1)^kS_kt^{n-k},$$
where $S_0=1$ by definition. For fixed $p\in M^n$, the spectral theorem allows us to choose on $T_pM$ an orthonormal basis $\{e_1,\ldots,e_n\}$ of eigenvectors of $A_p$, with corresponding
eigenvalues $\lambda_1,\ldots,\lambda_n$, respectively. One thus immediately sees that
$$S_r=\sigma_r(\lambda_1,\ldots,\lambda_n),$$
where $\sigma_r\in\mathbb R[X_1,\ldots,X_n]$ is the $r$-th elementary symmetric polynomial on the indeterminates
$X_1,\ldots,X_n$.

For $1\leq r\leq n$, one defines the $r$-th mean curvature $H_r$ of $x$ by
$${n\choose r}H_r=(-1)^rS_r=\sigma_r(-\lambda_1,\ldots,-\lambda_n).$$
One also let the $r$-th Newton transformation $P_r$ on $M^n$ be given by setting $P_0=\text{Id}$ and, for $1\leq r\leq n$, via the recurrence relation
\begin{equation}\label{eq:Newton operators}
P_r=(-1)^rS_r\text{Id}+AP_{r-1}.
\end{equation}

A trivial induction shows that
$$P_r=(-1)^r(S_r\text{Id}-S_{r-1}A+S_{r-2}A^2-\cdots+(-1)^rA^r),$$
so that Cayley-Hamilton theorem gives $P_n=0$. Moreover, since $P_r$ is a polynomial in $A$ for every $1\leq r\leq n$, it is also self-adjoint and commutes with $A$. Therefore, all bases of $T_pM$ diagonalizing $A$ at $p\in M^n$ also diagonalize all of the $P_r$ at $p$. If $\{e_1,\ldots,e_n\}$ is such a basis and $A_i$ denotes the restriction of $A$ to $\langle e_i\rangle^{\bot}\subset T_p\Sigma$, it is easy to see that
$$\det(t\text{Id}-A_i)=(-1)^kS_k(A_i)t^{n-1-k},$$
where
$$S_k(A_i)=\sum_{\stackrel{1\leq j_1<\ldots<j_k\leq n}{j_1,\ldots,j_k\neq i}}\lambda_{j_1}\cdots\lambda_{j_k}.$$
With the above notations, it is also immediate to check that $P_re_i=(-1)^rS_r(A_i)e_i$, and it is a standard fact that
\begin{enumerate}
\item[(i)] ${\rm tr}(P_r)=(-1)^r(n-r)S_r=b_rH_r$;
\item[(ii)] ${\rm tr}(AP_r)=(-1)^r(r+1)S_{r+1}=-b_rH_{r+1}$;
\item[(iii)] ${\rm tr}(A^2P_r)=(-1)^r(S_1S_{r+1}-(r+2)S_{r+2})$,
\end{enumerate}
where $b_r=(n-r){n\choose r}$.

Associated to each Newton transformation $P_r$ one has the second order linear differential operator $L_r:\mathcal
D(M)\rightarrow\mathcal D(M)$ given by
$$L_r(f)={\rm tr}(P_r\,\text{Hess}\,f).$$
In particular, $L_0=\Delta$, the Laplace operator on smooth functions on $M$. If $\overline M^{n+1}$ is of constant sectional curvature, H. Rosenberg~\cite{Rosenberg:93} proved that
$$L_r(f)={\rm div}(P_r\nabla f).$$

A variation of $x$ is a smooth mapping
$$X:M^n\times(-\epsilon,\epsilon)\rightarrow\overline M^{n+1}$$
satisfying the following conditions:
\begin{enumerate}
\item[(i)] For $t\in(-\epsilon,\epsilon)$, the map $X_t:M^n\rightarrow\overline M^{n+1}$
given by $X_t(p)=X(t,p)$ is a spacelike immersion such that $X_0=x$.
\item[(ii)] $X_t\big|_{\partial M}=x\big|_{\partial M}$, for all $t\in(-\epsilon,\epsilon)$.
\end{enumerate}

The variational field associated to the variation $X$ is the vector field $\frac{\partial X}{\partial t}\Big|_{t=0}$. Letting $N_t$ denote the unit normal vector field along $X_t$ and
$f=-\langle\frac{\partial X}{\partial t},N_t\rangle$, we get
\begin{equation}\label{eq:decomposition of variational vector field}\frac{\partial X}{\partial t}=fN_t+\left(\frac{\partial X}{\partial
t}\right)^{\top},\end{equation} where $\top$ stands for tangential
components.

Following~\cite{BdCE:88}, we set the balance of volume of the variation $X$ as the function $\mathcal V:(-\epsilon,\epsilon)\rightarrow\mathbb R$ given by
$$\mathcal V(t)=\int_{M\times[0,t]}X^*(d\overline M),$$
and we say that $X$ is volume-preserving if $\mathcal V$ is constant.

Letting $dM_t$ denote the volume element of the metric induced on $M$ by $X_t$, we recall the following standard result (cf.~\cite{Xin:03}).

\begin{lemma}\label{lemma:first variation}
Let $\overline M^{n+1}$ be a time-oriented Lorentz manifold and $x:M^n\rightarrow\overline M^{n+1}$ a closed spacelike hypersurface. If $X:M^n\times(-\epsilon,\epsilon)\rightarrow\overline M^{n+1}$ is a variation of $x$, then
$$\frac{d\mathcal V}{dt}=\int_MfdM_t.$$
In particular, $X$ is volume-preserving if and only if $\int_MfdM_t=0$ for all $t$.
\end{lemma}

We remark that Lemma 2.2 of~\cite{BdCE:88} remains valid in the Lorentz context, i.e., if $f_0:M\rightarrow\mathbb R$ is a smooth function such that $\int_Mf_0dM=0$, then there exists a volume-preserving variation of $M$ whose variational field is $f_0N$.

According to~\cite{Barbosa:97}, we define the $r$-area functional $\mathcal A_r:(-\epsilon,\epsilon)\rightarrow\mathbb R$ associated to the variation $X$ by
$$\mathcal A_r(t)=\int_MF_r(S_1,S_2,\ldots,S_r)dM_t,$$
where $S_r=S_r(t)$ and $F_r$ is recursively defined by setting $F_0=1$, $F_1=-S_1$ and, for $2\leq r\leq n-1$,
$$F_r=(-1)^rS_r-\frac{c(n-r+1)}{r-1}F_{r-2}.$$
In particular, if $r=0$, then $\mathcal A_0$ is the classical area functional.

The Lorentz analogue of Proposition 4.1 of~\cite{Barbosa:97} is stated in the following Lemma (for another proof, see Lemma 2.2 of~\cite{Camargo:09}).

\begin{lemma}\label{lemma:computing the time-derivative of Sr}
If $x:M^n\rightarrow\overline M^{n+1}_c$ is a closed spacelike hypersurface of the time-oriented Lorentz manifold $\overline M^{n+1}_c$ of constant sectional curvature $c$, and
$X:M^n\times(-\epsilon,\epsilon)\rightarrow\overline M^{n+1}_c$ a variation of $x$, then
\begin{equation}\label{eq:differentiation of S_r}
\frac{\partial S_{r+1}}{\partial t}=(-1)^{r+1}\left[L_{r}f+c{\rm tr}(P_{r})f-{\rm tr}(A^2P_{r})f\right]+\langle\left(\frac{\partial X}{\partial t}\right)^{\top},\nabla S_{r+1}\rangle.
\end{equation}
\end{lemma}

As in~\cite{Barbosa:97}, the previous lemma allows us to compute the first variation of the $r$-area functional, according to the following

\begin{proposition}\label{prop:first variation}
Under the hypotheses of Lemma~{\rm\ref{lemma:computing the time-derivative of Sr}}, we have
\begin{equation}\label{eq:first variation}
\mathcal A_r'(t)=\int_M[(-1)^{r+1}(r+1)S_{r+1}+c_r]f\,dM_t,
\end{equation}
where $c_r=0$ if $r$ is even and $c_r=-\frac{n(n-2)(n-4)\ldots(n-r+1)}{(r-1)(r-3)\ldots 2}(-c)^{(r+1)/2}$ if $r$ is odd.
\end{proposition}

In order to characterize spacelike immersions of constant $(r+1)-$th mean curvature, let $\lambda$ be a real constant and $\mathcal J_r:(-\epsilon,\epsilon)\rightarrow\mathbb R$ be the Jacobi functional associated to $X$, i.e.,
$$\mathcal J_r(t)=\mathcal A_r(t)-\lambda\mathcal V(t).$$
As an immediate consequence of (\ref{eq:first variation}) we get
$$\mathcal J_r'(t)=\int_M[b_rH_{r+1}+c_r-\lambda]fdM_t,$$
where $b_r=(r+1){n\choose r+1}$. Therefore, if we choose $\lambda=c_r+b_r\overline H_{r+1}(0)$, where
$$\overline H_{r+1}(0)=\frac{1}{\mathcal A_0(0)}\int_MH_{r+1}(0)dM$$
is the mean of the $(r+1)-$th curvature $H_{r+1}(0)$ of $M$, we arrive at
$$\mathcal J_r'(t)=b_r\int_M[H_{r+1}-\overline H_{r+1}(0)]fdM_t.$$
Hence, a standard argument (cf.~\cite{BdC:84}) shows that $M$ is a critical point of $\mathcal J_r$ for all variations of $x$ if and only if $M$ has constant $(r+1)-$th mean curvature.

As in~\cite{Barbosa:97}, we wish to study spacelike immersions $x:M^n\rightarrow\overline M^{n+1}$ that maximize $\mathcal J_r$ for all variations $X$ of $x$. The above dicussion shows that $M$ must have constant $(r+1)-$th mean curvature and, for such an $M$, leads us naturally to compute the second variation of $\mathcal J_r$. This, in turn, motivates the following

\begin{definition}
Let $\overline M^{n+1}_c$ be a Lorentz manifold of constant sectional curvature $c$, and $x:M^n\rightarrow\overline M^{n+1}$ be a closed spacelike hypersurface having constant $(r+1)-$th mean
curvature. We say that $x$ is strongly $r$-stable if, for every smooth function $f:M\rightarrow\mathbb R$ one has $\mathcal J_r''(0)\leq 0$.
\end{definition}

The sought formula for the second variation of $\mathcal J_r$ is another straightforward consequence of Proposition~\ref{prop:first variation}.

\begin{proposition}
Let $x:M^n\rightarrow\overline M^{n+1}_c$ be a closed spacelike hypersurface of constant $(r+1)-$mean curvature $H_{r+1}$. If $X:M^n\times(-\epsilon,\epsilon)\rightarrow\overline M^{n+1}_c$ is a variation of $x$, then
\begin{equation}\label{eq:second formula of variation}
\mathcal J_r''(0)=(r+1)\int_M\left[L_r(f)+c{\rm tr}(P_r)f-{\rm
tr}(A^2P_r)f\right]fdM.
\end{equation}
\end{proposition}

Back to the conformally stationary setting, in what follows we need a formula first derived in~\cite{Alias:07}. As stated below, it is the Lorentz version of the one stated and proved in~\cite{Barros:09}.

\begin{lemma}\label{lemma:Lr of conformal vector field}
Let $\overline M^{n+1}_c$ be a conformally stationary Lorentz manifold having constant sectional curvature $c$ and conformal vector field $V$. Let also $x:M^n\rightarrow\overline M^{n+1}_c$ be a spacelike hypersurface and $N$ a future-pointing, unit normal vector field globally defined on $M^n$. If $\eta=\langle V,N\rangle$, then
\begin{equation}\label{eq:Laplacian formula_I}
\begin{split}
L_r\eta&\,={\rm tr}(A^2P_r)\eta-c\,{\rm tr}(P_r)\eta-b_rH_rN(\psi)\\
&\,\,\,\,\,\,\,+b_rH_{r+1}\psi+\frac{b_r}{r+1}\langle V,\nabla
H_{r+1}\rangle,
\end{split}
\end{equation}
where $\psi:\overline M^{n+1}\rightarrow\mathbb R$ is the conformal factor of $V$, $H_j$ is the $j-$th mean curvature of $x$ and $\nabla H_j$ stands for the gradient of $H_j$ on $M$.
\end{lemma}

We are now in position to state and prove the following

\begin{theorem}\label{thm:r_estability}
Let $\overline M^{n+1}_c$ be a timelike geodesically complete conformally stationary Lorentz manifold of constant sectional curvature $c$ having a closed conformal timelike vector field $V$, and let $x:M^n\rightarrow\overline M^{n+1}_c$ be a closed, strongly $r-$stable spacelike hypersurface. Suppose that the conformal factor $\psi$ of $V$ satisfies the condition
$$\frac{H_r}{\sqrt{-\langle V,V\rangle}}\frac{\partial\psi}{\partial t}\geq\max\{H_{r+1}\psi,0\},$$
where $t\in\mathbb R$ denotes the real parameter of the flow of $V$. If the set where $\psi=0$ has empty interior in $M$, then $M^n$ is either $r-$maximal or a leaf of the foliation $V^{\bot}$.
\end{theorem}

\begin{proof}
Since $M^n$ is strongly $r$-stable, it follows from (\ref{eq:second formula of variation}) that
\begin{equation}\label{eq:auxiliar I para r_estability}
 (r+1)\int_M\left[L_r(f)+c{\rm tr}(P_r)f-{\rm tr}(A^2P_r)f\right]fdM\leq 0,
\end{equation}
for all smooth $f:M\rightarrow\mathbb R$. In particular, since $H_{r+1}$ is constant on $M$, taking $f=\eta=\langle V,N\rangle$ in (\ref{eq:Laplacian formula_I}), we get
$$L_r\eta+c\,{\rm tr}(P_r)\eta-{\rm tr}(A^2P_r)\eta=-b_rH_rN(\psi)+b_rH_{r+1}\psi,$$
so that (\ref{eq:auxiliar I para r_estability}) gives
\begin{equation}\label{eq:auxiliar para prova}
\int_M\left[-H_rN(\psi)+H_{r+1}\psi\right]\langle V,N\rangle dM\leq 0.
\end{equation}

However, it follows from (\ref{eq:gradient conformal factor}) that
$$N(\psi)=\langle N,\overline\nabla\psi \rangle=-\nu(\psi)\langle \nu,N\rangle=\frac{\partial\psi}{\partial t}\frac{\cosh\theta}{\sqrt{-\langle V,V\rangle}},$$
where $\theta$ is the hyperbolic angle between $V$ and $N$. Substituting the above into (\ref{eq:auxiliar para prova}), we finally arrive at
$$\int_M\left[H_r\frac{\partial\psi}{\partial t}\frac{\cosh\theta}{\sqrt{-\langle V,V\rangle}}-H_{r+1}\psi\right]\frac{\cosh\theta}{\sqrt{-\langle V,V\rangle}}dM\leq 0.$$

Arguing as in the end of the proof of Theorem 1.1 of~\cite{BBC:08}, we get
$$H_r\frac{\partial\psi}{\partial t}(\cosh\theta-1)=0\ \ \text{and}\ \ \frac{H_r}{\sqrt{-\langle V,V\rangle}}\frac{\partial\psi}{\partial t}=H_{r+1}\psi$$
on $M$. But since $H_{r+1}$ is constant on $M$, either $M$ is $r-$maximal or $H_{r+1}\neq 0$ on $M$. If this last case happens, the condition on the zero set of $\psi$ on $M$, together with the above, gives $H_r\frac{\partial\psi}{\partial t}\neq 0$ on a dense subset of $M$, and hence $\cosh\theta=1$ on this set. By continuity, $\cosh\theta=1$ on $M$, so that $M$ is a leaf of the foliation $\nu^{\bot}$.
\end{proof}

The following corollary is immediate.

\begin{corollary}
Let $x:M^n\rightarrow\mathbb S_1^{n+1}$ be a closed, strongly $r-$stable spacelike hypersurface, such that the set of points in which $M^n$ intersects the equator of $\mathbb S_1^{n+1}$ has empty interior in $M$. If $$H_r\geq\max\{(\sinh t)H_{r+1},0\},$$
then either $M^n$ is $r-$maximal or an umbilical round sphere.
\end{corollary}

\section*{Acknowledgements}
This work was started when the fourth author was visiting the Mathematics and Statistics Departament of the
Universidade Federal de Campina Grande. He would like to thank this institution for its hospitality.

\end{document}